\documentclass[12pt,reqno]{amsart}
\usepackage[mathlines]{lineno}

\usepackage{times}
\usepackage{amsfonts}
\usepackage{amssymb}
\usepackage{latexsym}
\usepackage{xcolor}

\newtheorem{theorem}{Theorem}
\newtheorem{lemma}[theorem]{Lemma}

\newtheorem{remark}{Remark}
\newtheorem{corollary}[theorem]{Corollary}

\def\nin{\relax\hbox{$/\kern-.7em{\rm \in\,}$}}

\begin{document}


\title[On a unique two-dimensional integral operator]
{On a unique two-dimensional integral operator homogeneous with respect to all orientation preserving linear transformations}

\vspace{10mm}

\author{Zhirayr Avetisyan *}
\address{Department of Mathematics: Analysis, Logic and Discrete Mathematics, University of Ghent, Krijgslaan 281, Building S8, B 9000 Ghent, Belgium $\&$ Regional Mathematical Center, Southern Federal
University, Rostov-on-Don, 344090, Russia}
\email{zhirayr.avetisyan@ugent.be}

\author{Alexey Karapetyants}
\address{Institute of Mathematics, Mechanics and Computer Sciences $\&$ Regional Mathematical Center, Southern Federal
University, Rostov-on-Don, 344090, Russia}
\email{karapetyants@gmail.com}

\author{Adolf Mirotin}
\address{Department of Mathematics and Programming Technologies, Francisk Skorina Gomel State University, Gomel, 246019, Belarus $\&$ Regional Mathematical Center, Southern Federal
University, Rostov-on-Don, 344090, Russia}
\email{amirotin@yandex.ru}

\keywords{Homogeneous integral operator, tensor product of spaces, Hilbert transform.\\
*Corresponding author}
\subjclass[2010]{}

\begin{abstract}
In this paper, we consider a two-dimensional operator with an antisymmetric integral kernel, recently introduced by Z. Avetisyan and A. Karapetyants in connection to the study of general homogeneous operators. This is the unique two-dimensional operator that has an antisymmetric kernel homogeneous with respect to all orientation-preserving linear transformations of the plane. It is shown that the operator under consideration interacts naturally, both in Cartesian and polar coordinates, with projective tensor products of some classical functional spaces, such as Lebesgue, Hardy, and H\"{o}lder spaces; conditions for their boundedness as operators acting from these spaces to Banach lattices of measurable functions and estimates of their norms are given.
\end{abstract}

\maketitle

\section{Introduction}\setcounter{equation}{0}

In \cite{AK} the unique (up to a constant factor) integral operator with
antisymmetric integral kernel homogeneous with respect to all orientation preserving linear transformations was introduced in the form
\begin{equation}\label{eq:UniqueIntOpn=2}
({\bf K}f)(x_1,x_2)=\frac{1}{\pi}\int_{\mathbb{R}}\int_{\mathbb{R}}\frac{f(y_1,y_2)}{x_1y_2-x_2y_1}dy_1dy_2.
\end{equation}
The convention we adopted in \eqref{eq:UniqueIntOpn=2} is that the Cauchy principal value applies to
each integral. What we can immediately see is that, at least formally, this operator maps an arbitrary function to a homogeneous function on $\mathbb{R}^2\setminus \{0\}$ of order $-1,$ in case the integral (\ref{eq:UniqueIntOpn=2}) exists.

The operator ${\bf K}$ via polar coordinates is related to the integral operator $\mathcal{K}$ acting on functions $\varphi=\varphi(r,\alpha)$ as follows,
\begin{equation}\label{eq:polarK}
(\mathcal{K}\varphi)(r,\alpha) = \frac{1}{\pi r} \int_0^\infty d\rho\int_{-\pi}^{\pi}\frac{\varphi(\rho,\theta)}{\sin(\theta-\alpha)}d\theta .
\end{equation}
As will become apparent from the discussion in the next section, polar coordinates are more relevant to this integral kernel from the conceptual point of view; it arises naturally in the punctured plane $\mathbb{R}^2\setminus\{0\}$, and its extension to the whole plane is the reason behind analytical subtleties such as the use of Cauchy principal values. Thus, the operator $\mathcal{K}$ represents an interest on its own
and, in a certain sense, better reflects certain features of both: for instance, it is evident that it vanishes on radial functions.

Thus, we have two formal integral operators, or two integral kernel formulae, related to each other by a coordinate transformation, see Section \ref{Sec3}. Of course, at the level of an integral kernel, these are merely two representations of the same kernel, and there are uncountably many other possible coordinates that could be applied, in principle. However, it is when we construct concrete integral operators, by choosing the appropriate domains of definition, that the particular coordinate representations become significant. Namely, we will see that each of our two coordinate representations is well suited for defining an integral operator on certain tensor products of classical function spaces. Thus, the formal integral operators $\bf K$ and $\mathcal{K}$ will be promoted to distinct integral operators on domains of tensor product structure according to their coordinate representations. Boundedness results will be obtained on algebraic tensor products, and then naturally extended to projective tensor products of Banach spaces. For the definition of a projective tensor product $F_1\hat\otimes_\pi F_2$ of Banach spaces $F_1, F_2$ see, e.g., \cite{Ryan}, or \cite{Helem}.

Integral operators with homogeneous kernels and equations with such operators have attracted the attention of many experts. In the one-dimensional case, this theory goes back to the works of Hardy-Littlewood-Pólya. Moreover, the one-dimensional case can be reduced to convolution operators by means of an exponential change of variable. But the situation is different in the case $n\geq 2:$ a large class of operators with homogeneous kernels cannot be reduced to convolution operators. A very detailed study of integral operators with homogeneous kernels in the classical sense is found in the monographs \cite{KarapetiantsSamko-book-rus,KarapetiantsSamko-book-eng}, see also the review paper \cite{KarapetiantsSamko}. The article \cite{AK} appeared in response to the desire to define and characterize homogeneous operators (not necessarily integral ones) and homogeneous integral kernels at an exhaustive level of abstraction, as the most general extension of the theory of integral operators with homogeneous kernels presented in the above-mentioned books.

After we have analyzed and understood the abstract properties of operators of the type (\ref{eq:UniqueIntOpn=2}), the question naturally arises about a more detailed study of such operators in specific spaces of functions suitable for such operators, which is the main goal of this work. This is done in Sections \ref{Sec4}, \ref{Sec5}. In addition, in this work, we want to separately and in detail clarify the origin of such operators; we have devoted Section \ref{Sec2} to this issue. Section \ref{Sec3} serves for writing integral representations for the operators $\bf K$, $\mathcal{K}.$


\section{The origin of the operator $\mathbf{K}$}\label{Sec2}
\setcounter{equation}{0}

At this point, let us step back a little and address the question of how this particular operator emerged and what its geometrical interpretation is. One commonplace approach to reducing arbitrariness in quantitative sciences consists in imposing certain covariance under a sufficiently large group of symmetries. Thus, in our case, one takes a homogeneous space $G/H$ of a group $G$ with a stabilizer subgroup $H\subset G$, and considers a measure $\mu$ on $G/H$ that is relatively invariant (also called pseudo-invariant) with respect to the $G$-action,
$$
\mu(g\cdot A)=\lambda_g\mu(A),\quad\forall g\in G,\quad\forall A\in\Sigma(G/H),
$$
where $\Sigma(G/H)$ is a certain $G$-invariant $\sigma$-algebra of measurable subsets, and
$$
G\ni g\overset{\lambda}{\longmapsto}\lambda_g\in\mathbb{R}_+
$$
is a real group character. Equivalently, one could begin with a measure space and a transitively acting subgroup of measure dilations, as thoroughly discussed in \cite{AK}. If the character $\lambda$ is trivial, $\lambda\equiv1$, then we are dealing with a $G$-invariant measure on $G/H$. More generally, the kernel of the character, $\ker\lambda\subset G$, is a normal subgroup. This leads the analysis to the following dichotomy: Case A, when $H\subset\ker\lambda$, and Case B, when $H\not\subset\ker\lambda$. It was demonstrated in \cite{AK} that in Case A, the measure $\mu$ can be slightly modified to yield a $G$-invariant measure on $G/H$, and the setting is essentially reduced to that of the standard harmonic analysis. It is in the less conventional Case B that a reduction to invariant measures is impossible. This case is characterized by the stabilizer subgroup $H$ being large relative to $G$, or in other words, the group of symmetries $G$ being excessively large for the space $G/H$.

Relatively invariant measures appear in a variety of mathematical contexts. Perhaps one of the most immediate such contexts is the (left or right) Haar measure $\mu$ on a locally compact topological group $X$ under the action of group automorphisms $\mathrm{Aut}(X)$. A further, lesser understood, and potentially rather useful such example is the Plancher\'el measure $\hat\mu$ on the unitary dual space $\hat X$ of a type I locally compact topological group $X$, under the dual action of the same group $\mathrm{Aut}(X)$ of automorphisms. In this case, the action is not transitive, so that one has to analyze each orbit separately.

If $L(G/H,\mu)$ is the space of $\Sigma(G/H)$-measurable functions on $G/H$ identified up to $\mu$-null subsets, the left quasi-regular representation $\operatorname{L}$ of $G$ on $L(G/H,\mu)$ is given by
$$
\operatorname{L}_gf(x)=f(g^{-1}x),\quad\forall g\in G,\quad\forall f\in L(G/H,\mu),\quad\mu\mbox{-a.e.}\;x\in G/H.
$$
For a vector subspace $D\subset L(G/H,\mu)$, a linear operator $\operatorname{K}:D\to L(G/H,\mu)$ is called homogeneous if
$$
\operatorname{L}_gD\subset D,\quad\operatorname{L}_g\operatorname{K}f=\operatorname{K}\operatorname{L}_gf,\quad\forall g\in G,\quad\forall f\in D.
$$
If the operator $\operatorname{K}$ is an integral operator with a measurable kernel $K\in L(G/H\times G/H,\mu^{\otimes2})$,
$$
\operatorname{K}f(x)=\int_{G/H}K(x,y)f(y)d\mu(y),\quad\mu\mbox{-a.e.}\;x\in G/H,\quad\forall f\in D,
$$
then, under natural assumptions on the expanse of $D$, it was shown (Theorem 3 in \cite{AK}) that $\operatorname{K}$ is homogeneous in the above sense if and only if the integral kernel $K$ is so-called weakly homogeneous, i.e.,
$$
(\forall g\in G)\quad\left(\mu^{\otimes2}\mbox{-a.e.}\;(x,y)\in G/H\times G/H\right)\quad K(gx,gy)=\frac1{\lambda_g}K(x,y).
$$
If the integral kernel $K$ is piecewise continuous, for example, then weak homogeneity implies the slightly stronger condition of strong homogeneity,
$$
\left(\mu^{\otimes2}\mbox{-a.e.}\;(x,y)\in G/H\times G/H\right)\quad(\forall g\in G)\quad K(gx,gy)=\frac1{\lambda_g}K(x,y).
$$
Strongly homogeneous (or simply, homogeneous) integral kernels were fully classified in Theorem 4 of \cite{AK} by building an explicit bijective correspondence between such kernels $K$ and functions $F$ that behave in a certain way with respect to the left $H$-action:
$$
K(aH,bH)=\frac{F(a^{-1}bH)}{\lambda_a},\quad\forall(aH,bH)=(x,y)\in G/H\times G/H,\quad F\in\mathcal{F}^\lambda_H,
$$
where
$$
\mathcal{F}^\lambda_H=\left\{F\in L(G/H,\mu)\;\vline\quad(\forall x\in G/H)\quad(\forall h\in H)\quad F(hx)=\frac{F(x)}{\lambda_h}\right\}.
$$
In particular, this classification implies that the dimension of the vector space of all homogeneous integral kernels equals the cardinality of the ``regular part'' of the double coset space $H\backslash G/H$. If $H$ is large, as discussed above, then this regular part consists of a few points or is empty, resulting in uniqueness or non-existence results for (non-trivial) homogeneous kernels.

It is essential to note that for topological and Lie groups, in Case B, homogeneous kernels $K$ are bound to be poorly controlled near infinity; the group $G$ is non-compact, because $H\not\subset\ker\lambda$ implies that $\lambda\not\equiv1$, whereas compact groups have only trivial characters. This becomes important when the homogeneous space $G/H$ arises as an open orbit in a larger space, with the embedding being such that infinity is mapped to a compact boundary. In such a situation, viewed from the perspective of the larger space, the kernel $K$ becomes singular.

Thus, from a carefully chosen combination of a homogeneous space $G/H$ and a relatively invariant measure $\mu$ emerges a unique homogeneous integral kernel $K$ with the respective homogeneous integral operator $\operatorname{K}$. Viewed from another perspective, for a given measure space $(M,\mu)$, the right choice of a transitively acting subgroup $G\subset\mathrm{Dil}(M,\mu)$ (if it exists) produces a unique $G$-homogeneous integral kernel $K$ and integral operator $\operatorname{K}$. One of the first examples that comes to one's mind is the punctured Euclidean space $M=\mathbb{R}^n\setminus\{0\}$ for $1<n\in\mathbb{N}$ with the Lebesgue measure $d\mu(x)=dx$. The natural candidates for the group of dilations are subgroups of linear transformations, $G\subset\mathrm{GL}(n)$, with the character being the determinant, $\lambda_g=\det g$, $\forall g\in G$. Taking the full group $G=\mathrm{GL}(n)$ is too generous; it was shown in \cite{AK} that for $n>2$ no non-trivial $\mathrm{GL}(n)$-homogeneous kernels exist, and for $n=2$ the unique (up to normalization) non-trivial $\mathrm{GL}(2)$-homogeneous kernel $K$ is strictly positive (or negative) and behaves like $\frac1r$ near the origin. While from the perspective of the punctured plane $\mathbb{R}^n\setminus\{0\}$ (which is topologically a cylinder) this is nothing extraordinary, when viewed as an integral operator acting on functions on $\mathbb{R}^2$, the corresponding integral operator $\operatorname{K}$ cannot be made to act on functions not vanishing at the origin, which is not satisfactory.

On the other hand, if we take instead only orientation-preserving linear transformations $G=\mathrm{GL}^+(n)$ (i.e., $\det g>0$), then for $n>2$ we still have no non-trivial kernels, but for $n=2$, aside from operators with strictly positive/negative kernels, we find a unique kernel which is antisymmetric, and therefore can be made sense of as an operator acting on functions over $\mathbb{R}^2$ with conditional convergence of the integral. That operator is exactly the subject of the present paper - our operator $\operatorname{K}={\bf K}$, up to a non-zero constant factor. As for $n>2$, in \cite{AK}, the question of finding suitable $G\subset\mathrm{GL}(n)$ to ensure unique existence was formulated as an open problem.

Therefore, the operator ${\bf K}$ emerges as the unique (up to normalization) integral operator with an anti-symmetric integral kernel homogeneous with respect to the group $G=\mathrm{GL}^+(2)$ of all orientation-preserving linear transformations of the plane. The anti-symmetry of the kernel $K$ allows one to define the action of the operator ${\bf K}$ on reasonably large domains by means of conditional convergence of the integral, as demonstrated in the forthcoming sections.


\section{On the integral operators $\bf K$, $\mathcal{K}$. }\label{Sec3}\setcounter{equation}{0}

The integral operator ${\bf K}$ acts as (\ref{eq:UniqueIntOpn=2}) for all $f$ for which the integral makes sense conditionally. Note that if we formally restrict to $x_2=1$ and $y_2=1$ (no integration over $y_2$) then we find the Hilbert transform, up to a constant factor. Thus, we can think of the operator ${\bf K}$ as an extension of the Hilbert transform to $\mathbb{R}^2$, which arises naturally as the unique (up to a constant factor) operator with an antisymmetric integral kernel homogeneous with respect to all orientation-preserving linear transformations.

The operator (\ref{eq:UniqueIntOpn=2}), after a change of variables inside the integrals, can be written in the form
\begin{eqnarray}\label{eq:UniqueIntOpn=2.a}
{\bf K}f(x_1,x_2)&=&-\frac{1}{\pi x_1}\int\limits_{\mathbb{R}}\frac{d\xi}{\frac{x_2}{x_1}-\xi}\int\limits_{\mathbb{R}}f(\eta,\xi\eta)d\eta\\&=&
-\frac{1}{\pi x_1}\int\limits_{\mathbb{R}}\left(\frac{1}{\frac{x_2}{x_1}-\xi}\frac{1}{\sqrt{1+\xi^2}}\int\limits_{\mathcal{L}_\xi}f(s)dl(s)\right)d\xi,
\nonumber
\end{eqnarray}
where $$\mathcal{L}_\xi=\{(t_1,t_2)\in\mathbb{R}^2: t_1\in\mathbb{R}, t_2=\xi t_1\}$$ is the line in the plane depending on
$\xi\in\mathbb{R}$, and $dl$ is the length element. Therefore, at least formally, the operator ${\bf K}$ is a composition up to a
multiplier $-\frac{1}{x_1}$ of the Hilbert ($\mathcal{H}$) and Radon ($\mathcal{R}$) transforms,
\begin{equation}\label{eq:UniqueIntOpn=2.b}
{\bf K}f(x_1,x_2)=-\frac{1}{x_1}
\mathcal{H}\left[\frac{1}{\sqrt{1+\xi^2}}\left(\mathcal{R}f\right)(\mathcal{L}_\xi)\right]\left(\frac{x_2}{x_1}\right)\, .
\end{equation}

Another representation of the integral operator \eqref{eq:UniqueIntOpn=2} suggested by V. D. Stepanov in personal communications (in 2021) is as follows:
\begin{eqnarray*}\label{eq:UniqueIntOpn=2.c}
{\bf K}f(x_1,x_2)&=&\frac{1}{\pi}\int_{\mathbb{R}}\int_{\mathbb{R}}\frac{f(y_1,y_2)}{x_1y_2-x_2y_1}dy_1dy_2
\\&=&\frac{1}{\pi}\frac{1}{x_1 x_2}\int_{\mathbb{R}}\int_{\mathbb{R}}\frac{f(y_1,y_2)}{\frac{y_2}{x_2}-\frac{y_1}{x_1}}dy_1dy_2
\\&=&\frac{1}{\pi}\int_{\mathbb{R}}\int_{\mathbb{R}}\frac{f(x_1 t_1,x_2t_2)}{t_2-t_1}dt_1dt_2.
\end{eqnarray*}
Now if
\begin{eqnarray*}
f_{x_1,x_2}(\cdot, t_2) &=& f(x_1(\cdot),x_2t_2),\\
Hf(x_1,x_2;t_2) &=& p.v. \int_{\mathbb{R}} \frac{f_{x_1,x_2}(t_1, t_2)}{t_2-t_1} dt_1 ,
\end{eqnarray*}
then
\begin{eqnarray*}
{\bf K}f(x_1,x_2)&=&\frac{1}{\pi}\int_{\mathbb{R}}Hf(x_1,x_2;t_2)dt_2 .
\end{eqnarray*}

And, finally, the representation related to polar coordinates in the plane is of use as well. Assuming that $$y_1=\rho \cos\theta,\,\,\,\, y_2=\rho\sin\theta$$
and
$$x_1=r\cos\alpha,\,\,\,\, x_2=r\sin\alpha,$$ we come (at least formally)  to the integral operator $\mathcal{K}$ acting on functions $\varphi=\varphi(r,\alpha)$ as in (\ref{eq:polarK}), that is:
\begin{equation*}
\mathcal{K}\varphi(r,\alpha) = \frac{1}{\pi r} \int_0^\infty d\rho\int_{-\pi}^{\pi}\frac{\varphi(\rho,\theta)}{\sin(\theta-\alpha)}d\theta .
\end{equation*}
Here the function $\varphi$ is related to $f$ by
$$\varphi(\rho,\theta)=f(\rho \cos\theta,\rho\sin\theta).$$
Thus, at least formally, the connection between ${\bf K}$ and $\mathcal{K}$ is as follows:
$$
(S{\bf K}f)(r,\alpha)=\frac{1}{\pi r} \int_0^\infty d\rho\int_{-\pi}^{\pi}\frac{(Sf)(\rho, \theta)}{\sin(\theta-\alpha)}d\theta,
$$
where
$$
(Sf)(\rho, \theta):=f(\rho \cos\theta,\rho\sin\theta),
$$
i. e.,
$$
S{\bf K}= \mathcal{K}S.
$$
In other words $S$  intertwines  ${\bf K}$   and $\mathcal{K}$.
 Since $S$ is an injection, one can write
$$
{\bf K}=S^{-1}\mathcal{K}S.
$$

\section{Boundedness of the operator $\mathcal{K}$ in tensor products}\label{Sec4}\setcounter{equation}{0}

We start our study with the operator (\ref{eq:polarK}), which is in some sense friendlier to deal with; in fact, we begin with the following, even more basic operator,
\begin{equation}\label{eq:polarK_2}
\mathcal{K}_2\varphi(\alpha) = \frac{1}{\pi }\int_0^\infty d\rho\int_{-\pi}^{\pi}\frac{\varphi(\rho,\theta)}{\sin(\theta-\alpha)}d\theta .
\end{equation}

To examine the boundedness of the operator $\mathcal{K}_2$ we need some auxiliary results.

First note that on the functions of the type
$$\varphi_0(\rho,\theta)=\varphi_1\otimes\varphi_2(\rho,\theta)=\varphi_1(\rho)\varphi_2(\theta)$$
the operator $\mathcal{K}_2$ acts as
\begin{eqnarray}\label{eq:polarK2}
\mathcal{K}_2\varphi_0(\alpha) &=&
\left(\int_0^\infty \varphi_1(\rho) d\rho\right) \left(\frac{1}{\pi}\int_{-\pi}^{\pi}\frac{\varphi_2(\theta)}{\sin(\theta-\alpha)}d\theta\right).
\nonumber
\end{eqnarray}

Consider the operator ($\phi$ stands for a $2\pi$-periodic function)
\begin{eqnarray}\label{eq:K_2}
\mathcal{K}_1\phi(\alpha) &:=& \frac{1}{\pi}\int_{-\pi}^{\pi}\frac{\phi(\theta)}{\sin(\theta-\alpha)}d\theta\\ \nonumber
&=& -\frac{1}{\pi}\int_{\alpha+\pi}^{\alpha-\pi}\frac{\phi(\alpha-t)}{\sin t}dt\\
\nonumber
&=&
\frac{1}{\pi}\int_{-\pi}^{\pi}\frac{\phi(\alpha-t)}{\sin t}dt.
\end{eqnarray}
Here and below we consider the principal value of the integral
$$
\int_{-\pi}^{\pi}f(t)dt:=\lim_{\varepsilon\to 0+}\left(\int_{-\pi+\varepsilon}^{-\varepsilon}f(t)dt+\int_{\varepsilon}^{\pi-\varepsilon}f(t)dt\right)
$$
with singularities $t=0$ and $t=\pm\pi$.

\begin{lemma}\label{lem:K_1_in_L2} The operator $\mathcal{K}_1$ acts on functions $e_k:=\frac{1}{\sqrt{2\pi}}e^{ikt}$ as follows:
\begin{eqnarray*}
\mathcal{K}_1 e_k &=& -2i e_k,\,\,\,\,\,\,\, k = 2m+1, \, m\in\mathbb{Z},\\
\mathcal{K}_1 e_k &=& 0,\,\,\,\,\,\,\,\,\,\,\,\,\,\,\,\,\,\,\,\, k = 2m, \, m\in\mathbb{Z}.
\end{eqnarray*}

In particular, the operator  $\mathcal{K}_1$ is bounded on $L^2[-\pi,\pi]$ and $\|\mathcal{K}_1\|_{L^2\to L^2}=2$.
\end{lemma}

\begin{proof}
Given $e_k:=\frac{1}{\sqrt{2\pi}}e^{ikt}$ where $k\in \mathbf{Z}$, we have
\begin{eqnarray}\label{eq:K_1_phi}
\mathcal{K}_1 e_k(\alpha)&=&
\frac{1}{\pi}\int_{-\pi}^{\pi}\frac{e_k(\alpha-t)}{\sin t}dt\\ \nonumber
&=&\frac{1}{\pi\sqrt{2\pi}}\lim_{\varepsilon\to 0+}\left(\int_{-\pi+\varepsilon}^{-\varepsilon}\frac{e^{ik(\alpha-t)}}{\sin t}dt+\int_{\varepsilon}^{\pi-\varepsilon}\frac{e^{ik(\alpha-t)}}{\sin t}dt\right)\\ \nonumber
&=&\frac{1}{\pi}e_k(\alpha)\lim_{\varepsilon\to 0+}\left(\int_{-\pi+\varepsilon}^{-\varepsilon}\frac{e^{-ikt}}{\sin t}dt+\int_{\varepsilon}^{\pi-\varepsilon}\frac{e^{-ikt}}{\sin t}dt\right)\\ \nonumber
&=&\frac{1}{\pi}e_k(\alpha)\lim_{\varepsilon\to 0+}\left(\int_{\pi-\varepsilon}^{\varepsilon}\frac{e^{i kt}}{\sin t}dt+\int_{\varepsilon}^{\pi-\varepsilon}\frac{e^{-i kt}}{\sin t}dt\right)\\ \nonumber
&=&\frac{1}{\pi}e_k(\alpha)\lim_{\varepsilon\to 0+}\int_{\varepsilon}^{\pi-\varepsilon}\frac{e^{-i kt}-e^{i kt}}{\sin t}dt\\ \nonumber
&=&-\frac{2 i}{\pi}e_k(\alpha)\int_{0}^{\pi}\frac{\sin kt}{\sin t}dt.
\end{eqnarray}
Finally, $\mathcal{K}_1e_k=-2i e_k$ if $k$ is odd and $\mathcal{K}_1e_k=0$ if $k$ is even.

Since the sequence $(e_k)_{k\in \mathbf{Z}}$ is an orthonormal basis of $L^2[-\pi,\pi]$, the last statement follows.
\end{proof}

Now we are in a position to prove the following result \footnote{The projective tensor product $L^1(\mathbb{R}_+)\hat\otimes_\pi L^2[-\pi,\pi]$  can be
identified with the  Banach space $L^1(\mathbb{R}_+,L^2[-\pi,\pi])$ of  equivalence
classes of Bochner integrable functions $f:\mathbb{R}_+\to L^2[-\pi,\pi]$ \cite[p. 29]{Ryan}. }.

\begin{theorem}\label{K2_F} Let $F$ be a Banach lattice of $2\pi$-periodic measurable functions on $\mathbb{R}$ such that constants belong to $F$. Then
$\mathcal{K}_2$ is bounded as an operator between $L^1(\mathbb{R}_+)\hat\otimes_\pi L^2[-\pi,\pi]$ and $F$ and its norm does not exceed $2\|1\|_F$.

\end{theorem}

\begin{proof} Let $\varphi_1\in L^1(\mathbb{R}_+)$ and $\psi_1\in C[-\pi,\pi]$. Then for all $\alpha\in[-\pi,\pi]$
$$
\mathcal{K}_2(\varphi_1\otimes \psi_1)(\alpha)=\left(\int_0^\infty \varphi_1(\rho)d\rho\right) \mathcal{K}_1\psi_1(\alpha).
$$
Therefore Lemma \ref{lem:K_1_in_L2} yields that
\begin{eqnarray*}
|\mathcal{K}_2(\varphi_1\otimes \psi_1)(\alpha)|&\le &\int_0^\infty |\varphi_1(\rho)|d\rho|\mathcal{K}_1\psi_1(\alpha)|\\ \nonumber
&\le&2\|\varphi_1\|_{L^1}\|\psi_1\|_{L^2}.
\end{eqnarray*}
It follows that for an arbitrary function from  the algebraic tensor product $L^1(\mathbb{R}_+)\otimes C[-\pi,\pi]$ of the form $\varphi=\sum_{i=1}^n \varphi_i\otimes \psi_i$ with $\varphi_i\in L^1(\mathbb{R}_+)$ and $\psi_i\in C[-\pi,\pi]$  and for all $\alpha\in[-\pi,\pi]$ one has
\begin{eqnarray*}
|\mathcal{K}_2(\varphi)(\alpha)|\le 2\sum_{i=1}^n\|\varphi_i\|_{L^1}\|\psi_i\|_{L^2}.
\end{eqnarray*}
So, for every $\alpha\in[-\pi,\pi]$ the map
$\varphi\mapsto \mathcal{K}_2\varphi(\alpha)$ is a  bounded linear functional on $L^1(\mathbb{R}_+)\otimes C[-\pi,\pi]$ as a subspace of $L^1(\mathbb{R}_+)\hat\otimes_\pi L^2[-\pi,\pi]$ and its norm does not exceed $2$. In turn, it follows that
$\mathcal{K}_2$ is bounded as an operator between $L^1(\mathbb{R}_+)\otimes C[-\pi,\pi]$ and $F$ and its norm does not exceed $2\|1\|_F$.
Since $L^1(\mathbb{R}_+)\otimes C[-\pi,\pi]$ is dense in $L^1(\mathbb{R}_+)\hat\otimes_\pi C[-\pi,\pi]$ (here we consider $C[-\pi,\pi]$ as a normed subspace of $ L^2[-\pi,\pi]$) and the latter is dense in $L^1(\mathbb{R}_+)\hat\otimes_\pi L^2[-\pi,\pi]$ (see, e.g., \cite[Corollary 13.4, p. 163]{DF}),
 the result follows.
\end{proof}

Let $\Lambda_\gamma$ stand for the H\"{o}lder space of $2\pi$-periodic functions on $\mathbb{R}$ with exponent  $\gamma$ ($0<\gamma<1$), and let $\|\cdot\|_{\Lambda_\gamma}$ stand for the corresponding semi-norm.

We will consider the Hilbert transform in the following form.
\begin{equation}\label{HT-1}
\mathcal{H}\phi(\alpha)=\frac{1}{2\pi}\int_{-\pi}^{\pi}\frac{\phi(\alpha-t)}{\tan \frac{t}{2}}dt
\end{equation}

Although, in the literature sometimes
it is more convenient to call
\begin{equation}\label{HT-2}
\mathrm{H}\phi(\alpha)=\frac{1}{\pi}\int_{0}^{\pi}\frac{\phi(\alpha-t)}{\alpha-\tau}d\tau
\end{equation} the Hilbert transform,
the difference between (\ref{HT-1}) and (\ref{HT-2}) is a convolution operator with a continuous kernel, see \cite{G-1981}, page 105.

\begin{lemma}\label{lm:K_2} The integral in \eqref{eq:K_2} exists for $\phi$ in the H\"{o}lder space $\Lambda_\gamma$ of $2\pi$-periodic functions on $\mathbb{R}$ with exponent  $\gamma$ ($0<\gamma<1$).
\end{lemma}

\begin{proof} Indeed,
\begin{eqnarray}\label{eq:K_21}
\mathcal{K}_1\phi(\alpha)&=&\frac{1}{\pi}\int_{-\pi}^{\pi}\frac{\phi(\alpha-t)}{\sin t}dt\\ \nonumber
&=&\frac{1}{\pi}\int_{-\pi}^{\pi}\frac{\phi(\alpha-t)}{\frac{2\tan t/2}{1+\tan^2 t/2}}dt=\mathcal{H}\phi(\alpha)+\mathcal{J}\phi(\alpha) ,
\end{eqnarray}
where the Hilbert transform
$
\mathcal{H}\phi
$
exists for $\phi\in \Lambda_\gamma$ by the Privalov theorem (see, e.g., \cite[Section 6.16]{King}), and
\begin{equation}\label{eq:J}
\mathcal{J}\phi(\alpha) := \frac{1}{2\pi}\int_{-\pi}^{\pi}\phi(\alpha-t)\tan\frac{t}{2}dt.
\end{equation}

Since $$\int_{-\pi}^{\pi}\tan(t/2)dt=0$$
and  $\phi(\alpha-\pi)=\phi(\alpha+\pi)$, one can rewrite the last expression as follows:
\begin{eqnarray}\label{eq:J1}
\mathcal{J}\phi(\alpha) &= &\frac{1}{2\pi}\int_{-\pi}^{\pi}(\phi( \alpha-t)-\phi(\alpha-\pi))\tan\frac{t}{2}dt\\ \nonumber
&=&\frac{1}{2\pi}\int_{0}^{\pi}(\phi( \alpha-t)-\phi(\alpha-\pi))\tan\frac{t}{2}dt\\ \nonumber
&+&\frac{1}{2\pi}\int_{-\pi}^{0}(\phi( \alpha-t)-\phi(\alpha+\pi))\tan\frac{t}{2}dt.
\end{eqnarray}

To examine the behavior of the integral \eqref{eq:J} near the singularity $t=\pi$, consider the first term in the representation \eqref{eq:J1}.
We have
$$
|\phi(\alpha-t)-\phi(\alpha-\pi)|\le \|\phi\|_{\Lambda_\gamma}|t-\pi|^\gamma.
$$
Since $|\tan t/2|\le 2/(\pi-|t|)$, for $|t|<\pi$, this yields for $t\in [0,\pi)$
\begin{equation}\label{eq:est-}
|(\phi(\alpha-t)-\phi(\alpha-\pi))\tan t/2|\le 2\|\phi\|_{\Lambda_\gamma}(\pi-t)^{\gamma-1}.
\end{equation}
So, the integral \eqref{eq:J} converges at $t=\pi$.

The case of  the singularity $t=-\pi$ can be  considered in a similar way by means of the representation \eqref{eq:J1} and the following inequality, in which $t\in (-\pi,0]$:
\begin{equation}\label{eq:est+}
|(\phi(\alpha-t)-\phi(\alpha+\pi))\tan t/2|\le 2\|\phi\|_{\Lambda_\gamma}(\pi+t)^{\gamma-1}.
\end{equation}
This finishes the proof.
\end{proof}

\begin{corollary}\label{cor:K2alpha} For all even $\phi\in \Lambda_\gamma$ and  for all $\alpha\in[-\pi,\pi]$
\begin{equation}\label{Est-K-alpha}
|\mathcal{K}_1\phi(\alpha)|\le c_\gamma \|\phi\|_{\Lambda_\gamma}
\end{equation}
where $c_\gamma=\gamma^{-1} \pi^{\gamma-1}+ \|\mathcal{H}\|_{\Lambda_\gamma\to \Lambda_\gamma}(2\pi)^\gamma$.
\end{corollary}

\begin{proof}
Applying the estimates \eqref{eq:est-} and \eqref{eq:est+} to  \eqref{eq:J1} we get
\begin{eqnarray}\label{eq:J3}
|\mathcal{J}\phi(\alpha)| &\le&\frac{1}{2\pi}\int_{0}^{\pi}|\phi(\alpha-t)-\phi(\alpha-\pi)\tan\frac{t}{2}|dt\\ \nonumber
&+&\frac{1}{2\pi}\int_{-\pi}^{0}|\phi( \alpha-t)-\phi(\alpha+\pi))\tan\frac{t}{2}|dt\\ \nonumber
&\le&\frac{1}{2\pi}2\|\phi\|_{\Lambda_\gamma}\left(\int_{0}^{\pi}(\pi-t)^{\gamma-1}dt+\int_{-\pi}^{0}(\pi+t)^{\gamma-1}dt\right)\\ \nonumber
&=&\frac{1}{\gamma \pi^{1-\gamma}}\|\phi\|_{\Lambda_\gamma}.
\end{eqnarray}
Finally, note that the Hilbert transform  is bounded in $\Lambda_\gamma$ by Privalov theorem \cite[Section 6.16]{King}   and $\mathcal{H}\phi(0)=0$ for even $\phi$.
Then
$$
|\mathcal{H}\phi(\alpha)|=|\mathcal{H}\phi(\alpha)-\mathcal{H}\phi(0)|\le \|\mathcal{H}\phi\|_{\Lambda_\gamma}|\alpha|^\gamma\le \|\mathcal{H}\|_{\Lambda_\gamma\to \Lambda_\gamma}\|\phi\|_{\Lambda_\gamma}(2\pi)^\gamma,
$$
and thus the result follows from the equality \eqref{eq:K_21}.
\end{proof}

Below $\Lambda_\gamma^e$ denotes the space of even $2\pi$-periodic H\"{o}lder functions on $\mathbb{R}$ with exponent  $0<\gamma<1$.

\begin{corollary}\label{cor:K1} Let $F$ be a Banach lattice of $2\pi$-periodic functions on $\mathbb{R}$,  such that constants belong to $F$. Then $\mathcal{K}_1$ is a bounded operator between $\Lambda_\gamma^e$ and $F$ and
$$\|\mathcal{K}_1\|_{\Lambda_\gamma^e\to F}\le c_\gamma\|1\|_F.
$$
\end{corollary}

\begin{proof} Follows from the estimate (\ref{Est-K-alpha}).
\end{proof}

\begin{corollary}\label{cor:J} Let $F$ be a Banach lattice of $2\pi$-periodic functions on $\mathbb{R}$,  such that constants belong to $F$. Then $\mathcal{J}$ is a bounded operator between $\Lambda_\gamma$ and $F$ and $\|\mathcal{J}\|_{\Lambda_\gamma\to F}\le 1/\gamma\pi^{1-\gamma}\|1\|_F$.
\end{corollary}

\begin{proof} Follows from the  estimate \eqref{eq:J3}.
\end{proof}

Now we are in a position to prove the following result.

\begin{theorem}\label{L_Lambda} For every $\alpha\in[-\pi,\pi]$ the map
$\varphi\mapsto \mathcal{K}_2\varphi(\alpha)$ is a  bounded linear functional on $L^1(\mathbb{R}_+)\hat\otimes_\pi \Lambda_\gamma^e$  and its norm does not exceed $c_\gamma$.
\end{theorem}

\begin{proof} Let $\varphi_1\in L^1(\mathbb{R}_+)$ and $\varphi_2\in \Lambda_\gamma^e$. Then for all $\alpha\in[-\pi,\pi]$
$$
\mathcal{K}_2(\varphi_1\otimes \varphi_2)(\alpha)=\left(\int_0^\infty \varphi_1(\rho)d\rho\right) \mathcal{K}_1\varphi_2(\alpha).
$$
Therefore Corollary \ref{cor:K2alpha} yields that
\begin{eqnarray*}
|\mathcal{K}_2(\varphi_1\otimes \varphi_2)(\alpha)|&\le &\int_0^\infty |\varphi_1(\rho)|d\rho|\mathcal{K}_1\varphi_2(\alpha)|\\ \nonumber
&\le&c_\gamma\|\varphi_1\|_{L^1}\|\varphi_2\|_{ \Lambda_\gamma}.
\end{eqnarray*}
It follows that for an arbitrary function $\varphi=\sum_{i=1}^n \varphi_i\otimes \psi_i$ from  $L^1(\mathbb{R}_+)\otimes\Lambda_\gamma^e$ and for all $\alpha\in[-\pi,\pi]$ one has
\begin{eqnarray*}
|\mathcal{K}_2(\varphi)(\alpha)|\le c_\gamma\sum_{i=1}^n\|\varphi_i\|_{L^1}\|\psi_i\|_{ \Lambda_\gamma}.
\end{eqnarray*}
So, for every $\varphi\in L^1(\mathbb{R}_+)\hat\otimes_\pi  \Lambda_\gamma$ and for all $\alpha\in[-\pi,\pi]$
\begin{eqnarray}
|\mathcal{K}_2(\varphi)(\alpha)|\le c_\gamma\|\varphi|_{L^1\hat\otimes_\pi  \Lambda_\gamma}
\end{eqnarray}
and the result follows.
\end{proof}

\begin{corollary}\label{K_Lambda} For every $\alpha\in[-\pi,\pi]$ and for every $r>0$  the map
$\varphi\mapsto \mathcal{K}\varphi(r,\alpha)$ is a  bounded linear functional on $L^1(\mathbb{R}_+)\hat\otimes_\pi \Lambda_\gamma^e$  and its norm does not exceed $\frac1r c_\gamma$.
\end{corollary}

\begin{corollary}\label{L_Lambda} Let $F$ be a Banach lattice of $2\pi$-periodic measurable functions on $\mathbb{R}$ such that constants belong to $F$. Then
$\mathcal{K}_2$ is bounded as an operator between $L^1(\mathbb{R}_+)\hat\otimes_\pi \Lambda_\gamma^e$ and $F$ and its norm does not exceed $c_\gamma$.
\end{corollary}

\section{Boundedness of the operator ${\bf K}$ in tensor products}\label{Sec5}\setcounter{equation}{0}

In the following for a measurable subset $B$ of  $\mathbb{R}$ we consider $L^p(B)$ as a subspace of $L^p(\mathbb{R})$.

\begin{theorem}\label{Lp} Let $E$ be a Banach lattice of measurable functions on $\mathbb{R}^2$, $p\in (1,\infty)$, $1/p+1/q=1$ and $B$ a measurable subset of  $\mathbb{R}$.

(i) If the function $v_p(x_1,x_2)=|x_1|^{-1/q}|x_2|^{-1/p}$ belongs to $E$ then ${\bf K}$ is bounded as an operator between $L^q(\mathbb{R})\hat\otimes_\pi L^p(B)$ and $E$ and
$$
\|{\bf K}\|_{L^q(\mathbb{R})\hat\otimes_\pi L^p(B)\to E}\le  C_p\|v_p\|_E,
$$
where
$$
C_p=
\begin{cases}
\tan\frac{\pi}{2p}, 1<p\le 2,\\
\cot\frac{\pi}{2p}, 2<p<\infty
\end{cases}
$$
stands for the Riesz constant.

(ii) If $p<q$, the subset $B$ is bounded,  and ${\bf K}$ maps $L^q(\mathbb{R})\hat\otimes_\pi L^p(B)$ into  $E$, then $v_p$ belongs to $E$.
\end{theorem}

\begin{proof} (i) Let $f\in L^q(\mathbb{R})\hat\otimes_\pi L^p(B)$ be of the form $f=f_1\otimes f_2$ where $f_1\in  L^q(\mathbb{R})$,  $f_2\in  L^p(B)$. Then for $x_1\ne 0$ one has
\begin{eqnarray}\label{est1}
({\bf K}f)(x_1,x_2)&=&\frac{1}{\pi}\int_{\mathbb{R}}f_1(y_1)dy_1\int_{\mathbb{R}}\frac{f_2(y_2)}{x_1y_2-x_2y_1}dy_2\\ \nonumber
&=&-\frac{1}{x_1}\int_{\mathbb{R}}f_1(y_1)dy_1\frac{1}{\pi}\int_{\mathbb{R}}\frac{f_2(y_2)}{\frac{x_2}{x_1}y_1-y_2}dy_2\\ \nonumber
&=&-\frac{1}{x_1}\int_{\mathbb{R}}f_1(y_1)(Hf_2)\left(\frac{x_2}{x_1}y_1\right)dy_1
\end{eqnarray}
where $Hf_2$ denotes the Hilbert transform of $f_2$. Therefore the H$\ddot{\rm o}$lder inequality yields for $x_2\ne 0$
\begin{eqnarray*}
|({\bf K}f)(x_1,x_2)|&\le& \frac{1}{|x_1|}\|f_1\|_q\|Hf_2\left(\frac{x_2}{x_1}\cdot\right)\|_p\\ \nonumber
&=&\frac{1}{|x_1|}\frac{1}{\left|\frac{x_2}{x_1}\right|^{1/p}}\|f_1\|_q\|Hf_2\|_p=\frac{\pi}{|x_1|^{1/q}|x_2|^{1/p}}\|f_1\|_q\|Hf_2\|_p\nonumber
\end{eqnarray*}
(here and below $\|\cdot\|_p$ denotes the $L^p$ norm).
Applying the Riesz inequality we  get for all $x_1, x_2\ne 0$
\begin{equation}\label{est}
|({\bf K}f)(x_1,x_2)|\le \frac{C_p}{|x_1|^{1/q}|x_2|^{1/p}}\|f_1\|_q\|f_2\|_p.
\end{equation}
Every function $f\in L^q(\mathbb{R})\otimes L^p(B)$ is of the form
$$
f(y_1,y_2)=\sum_{i=1}^n f_i \otimes g_i(y_1,y_2)=\sum_{i=1}^n f_i(y_1)  g_i(y_2)
$$
with $f_i\in L^q(\mathbb{R})$, $g_i\in L^p(B)$.

Application of \eqref{est} implies
\begin{equation}\label{est2}
|({\bf K}f)(x_1,x_2)|\le \sum_{i=1}^n |{\bf K}(f_i \otimes g_i)(x_1,x_2)| \le \frac{C_p}{|x_1|^{1/q}|x_2|^{1/p}}\sum_{i=1}^n\|f_i\|_q\|g_i\|_p.
\end{equation}
So the definition of the projective norm yields that
\begin{equation}\label{est3}
|({\bf K}f)(x_1,x_2)|\le \frac{C_p}{|x_1|^{1/q}|x_2|^{1/p}}\|f\|_{ L^q(\mathbb{R})\hat\otimes_\pi L^p(B)}.
\end{equation}
Finally, since $E$ is  a Banach lattice of measurable functions, the inequality \ref{est3} implies that for all $f\in L^q(\mathbb{R})\hat\otimes_\pi L^p(B)$ we have
$$
\|{\bf K}f\|_E\le C_p\|v_p\|_E\|f\|_{ L^q(\mathbb{R})\hat\otimes_\pi L^p(B)}.
$$

(ii) Now let $p<q$, ${\bf K}$ maps $L^q(\mathbb{R})\hat\otimes_\pi L^p(B)$ into  $E$, and  $B$ is a bounded subset of $\mathbb{R}$. Then the function
$$
f_2(x_2):=\mathrm{sgn}(x_2)|x_2|^{-1/q}
$$
belongs to $L^p(B)$. Further,  it is known that
$$
(Hf_2)(y)=-\tan\frac{\pi}{2p}|y|^{-1/q}
$$
(see, e. g., \cite[p. 176 (30)]{BE}). If we put $f=\chi_{(0,1)}\otimes f_2$, then $f\in L^q(\mathbb{R})\hat\otimes_\pi L^p(B)$ and formula \eqref{est1} yield that
$$
({\bf K}f)(x_1,x_2)=\tan\frac{\pi}{2p}\frac{1}{x_1}\int_0^1\left|\frac{x_2}{x_1}y_1\right|^{-1/q}dy_1=p\tan\frac{\pi}{2p}x_1^{-1}|x_1|^{1/q}|x_2|^{-1/q}.
$$
 Since this function belongs to a lattice $E$, the function $v_p$ belongs to $E$, too.
\end{proof}

\begin{corollary} Let a non-negative measure $\mu$ on $\mathbb{R}^2$ be such that $v_p\in L^r(\mathbb{R}^2,\mu)$, $r\ge 1$. Then ${\bf K}$ is bounded as an operator between $L^q(\mathbb{R})\hat\otimes_\pi L^p(\mathbb{R})$ and the space $E= L^r(\mathbb{R}^2,\mu)$ and
$$
\|{\bf K}\|_{L^q(\mathbb{R})\hat\otimes_\pi L^p(\mathbb{R})\to L^r(\mathbb{R}^2,\mu)}\le  C_p\|v_p\|_{L^r(\mathbb{R}^2,\mu)}.
$$
\end{corollary}

\begin{corollary} Let $v_2\in E$. Then ${\bf K}$ is bounded as an operator between $L^2(\mathbb{R}^2)$ and $E$ and
$$
\|{\bf K}\|_{L^2(\mathbb{R}^2)\to E}\le  \|v_2\|_{E}.
$$
\end{corollary}
\begin{proof} This follows from the inequality \eqref{est2} with $p=2$ and the fact that $L^2(\mathbb{R}^2)=L^2(\mathbb{R})\dot\otimes L^2(\mathbb{R})$  (the Hilbert tensor product) see, e.g., \cite[Chapter 2, \S 8]{Helem}.
\end{proof}

\begin{theorem}\label{L_H_1} Let $E$ be a Banach lattice of measurable functions on $\mathbb{R}^2$. If the function $w(x_1,x_2)=|x_2|^{-1}$ belongs to $E$ then ${\bf K}$ is bounded as an operator between $L^\infty(\mathbb{R})\hat\otimes_\pi H^1(\mathbb{R})$ and $E$ and
$$
\|{\bf K}\|_{L^\infty(\mathbb{R})\hat\otimes_\pi H^1(\mathbb{R})\to E}\le  \|w\|_E.
$$
\end{theorem}

\begin{proof}  Let $f\in L^\infty(\mathbb{R})\hat\otimes_\pi H^1(\mathbb{R})$ be of the form $f=f_1\otimes f_2$ where $f_1\in  L^\infty(\mathbb{R})$,  $f_2\in  H^1(\mathbb{R})$.
Then the equality \eqref{est1} implies for $x_1,x_2\ne 0$
\begin{eqnarray*}
|({\bf K}f)(x_1,x_2)|&\le& \frac{1}{|x_1|}\|f_1|_\infty\|Hf_2\left(\frac{x_2}{x_1}\cdot\right)\|_1\\ \nonumber
&=&\frac{1}{|x_1|}\frac{1}{\left|\frac{x_2}{x_1}\right|}\|f_1|_\infty\|Hf_2\|_1\\
&\le& \frac{1}{|x_2|}\|f_1|_\infty\|f_2\|_{H^1},
\end{eqnarray*}
since $\|Hg\|_1\le \|g\|_{H^1}$ for $g\in H^1(\mathbb{R})$ (see, e.g., \cite[p. 222]{Krantz}). Now as in the proof of Theorem \ref{Lp} one get that for $f=\sum_{i=1}^n f_i\otimes g_i$ where $f_i\in  L^\infty(\mathbb{R})$,  $g_i\in  H^1(\mathbb{R})$ the next inequality holds true
$$
|({\bf K}f)(x_1,x_2)|\le  \frac{1}{|x_2|}\sum_{i=1}^n\|f_i|_\infty\|g_i\|_{H^1}
$$
and the result follows.
\end{proof}

\begin{corollary}\label{H_1_L} Let $F$ be a Banach lattice of measurable functions on $\mathbb{R}^2$.
If the function $w_1(x_1,x_2)=|x_1|^{-1}$ belongs to $F$ then ${\bf K}$ is bounded as an operator between $H^1(\mathbb{R})\hat\otimes_\pi L^\infty(\mathbb{R})$ and $F$ and
$$
\|{\bf K}\|_{H^1(\mathbb{R})\hat\otimes_\pi L^\infty(\mathbb{R})\to F}\le  \|w_1\|_F.
$$
\end{corollary}

\begin{proof} The map $g(x_1,x_2)\mapsto g(x_2,x_1)$ is an isometric isomorphism of the space $F$ on some  Banach lattice $E$ with $w_2\in E$. It remains to note that the  map $$f_1(y_1)f_2(y_2)\mapsto f_2(y_1)f_1(y_2)$$ extends to an isometric isomorphism of the space $L^\infty(\mathbb{R})\hat\otimes_\pi H^1(\mathbb{R})$ onto the space $H^1(\mathbb{R})\hat\otimes_\pi L^\infty(\mathbb{R})$.
\end{proof}

\begin{remark} As was noticed in \cite{AK},  image of the operator ${\bf K}$ consists of homogeneous functions of
order $-1$, so ${\bf K}$ cannot act, say, into Lebesgue  spaces   $L^p(\mathbb{R}^2)$ (e.g., into the space $L^1(\mathbb{R}^2) = L^1(\mathbb{R})\hat\otimes_\pi L^1(\mathbb{R})$).
\end{remark}

\vspace{.5cm}

\noindent {\bf Funding information.}
Alexey Karapetyants and Adolf Mirotin were supported by the Regional Mathematical Center of Southern Federal University under the program of the Ministry of Education and Science of the Russian Federation, agreement No. 075-02-2025-1720. Zhirayr Avetisyan was supported by the FWO Senior Research Grant G022821N and the Methusalem programme of the Ghent University Special Research Fund (BOF) (Grant number 01M01021).

\noindent
{\bf{Data availability and conflict of interest.}}
The authors confirm that all data generated or analyzed during this study are included in this article.
This work does not have any conflicts of interest.

\end{document}